\documentclass{article}
\usepackage{graphicx} % Required for inserting images
\usepackage{amsmath,amssymb,amsmath,amsthm}
\usepackage[table]{xcolor}
\title{Signatures of TQFTs and trace fields of two-bridge knots}
\author{Julien Marché}
\date{}
\newcommand{\Z}{\mathbb{Z}}
\newcommand{\Q}{\mathbb{Q}}
\newcommand{\C}{\mathbb{C}}
\newcommand{\R}{\mathbb{R}}
\newcommand{\SL}{\mathrm{SL}}
\renewcommand{\epsilon}{\varepsilon}
\DeclareMathOperator{\sg}{Sign}
\DeclareMathOperator{\tr}{Tr}
\DeclareMathOperator{\Res}{Res}

\DeclareMathOperator{\Sign}{Sign}
\DeclareMathOperator{\End}{End}
\DeclareMathOperator{\re}{Re}
\DeclareMathOperator{\Mod}{Mod}
\newcommand{\PU}{{\rm{PU}}}
\newtheorem{theorem}{Theorem}
\newtheorem{lemme}{Lemma}
\newtheorem{proposition}{Proposition}
\theoremstyle{remark}
\newtheorem{remark}{Remark}

\begin{document}
\maketitle
\begin{abstract}
Let $0<s<r$ be coprime odd integers. We show that the Frobenius algebras governing the signatures of SO$_3$ TQFTs at the root $q=\exp(i\pi s /r)$ contain (and are often equal to) the trace field of the two-bridge knot of parameters $(r,s)$. This gives an intriguing relationship between these two a priori unrelated objects of low-dimensional topology.
\end{abstract}

\section{Introduction}

Given $q\in \C$, a root of unity of order $2r$ with $r$ odd, we consider the family of SU$_2$/SO$_3$ quantum representations of mapping class groups of surfaces indexed by $\pm q$, as constructed for instance in \cite{BHMV}. Take $\Lambda=\{0,1,\ldots,r-2\}$ and $\Lambda_+=\Lambda\cap2\Z$. These representations take the following form for any closed compact oriented surface $S_{g,n}$ of genus $g$ with $n$ marked points:
$$ \rho_q^\lambda:\Mod(S_{g,n})\to \PU(W_q(S_{g,n},\lambda)).$$
Here $\Mod(S_{g,n})$ denote the mapping class group of $S_{g,n}$ fixing the marked points and $\lambda\in \Lambda^n$ ($\Lambda_+^n$ in the SO$_3$-case) is a coloring of the marked points. The Hermitian vector space $W_q(S,\lambda)$ is a modular functor: this means that it enjoys many compatibility properties with respect to usual operations on surfaces, as cutting along simple curves, see for instance \cite[Section 4]{DM}.

As a function of $q$, the dimension of $W_q(S_{g,n},\lambda)$ only depends on $r$: it is given by the celebrated Verlinde formula. On the contrary, its signature is a subtle function of $q$ not much studied until the recent work of B. Deroin and the author, see \cite{DM}. One of its main results is that cohomological invariants of the above representations, amongst which the signature appears at $0$-th order, form a semi-simple cohomological field theory. 

In this article, we consider only signatures: in particular, the present article is completely elementary and does not require any expertise in TQFT. The result of \cite{DM} states that signatures are part of a 1+1 semi-simple TQFT, or equivalently are governed by a Frobenius algebra, denoted by $V_q$ ($V_q^+$ for the SO$_3$ case). 

Let us give the precise statement: take $V_q$ (resp. $V_q^+$) to be the formal $\Q$-vector space with basis $\Lambda$ (resp. $\Lambda_+$). We denote by $e_\lambda$ the corresponding basis and define $\epsilon(e_\lambda)=1$ if $\lambda=0$ and $\epsilon(e_\lambda)=0$ otherwise. We will put in Section \ref{FrobSU2} a structure of semi-simple commutative algebra on $V_q$ such that the bilinear form $\eta(x,y)=\epsilon(xy)$ is non-degenerate on $V_q$ and the following holds:
$$\sg(W_q(S_{g,n},\lambda))=\epsilon(\Omega^g e_{\lambda_1}\cdots e_{\lambda_n}).$$ 

Here, we have written $\Omega=\sum_{i} x_iy_i\in V_q$ where $\eta^{-1}=\sum x_i\otimes y_i\in V_q\otimes V_q$. The same statement holds for SO$_3$: just replace $V_q$ with the subalgebra $V_q^+$ generated by $e_\lambda$ for $\lambda\in \Lambda_+$ and $\Omega$ with $\Omega^+=\frac{1}{2}\Omega$. 

Although explicit formulas were provided in \cite{DM}, the structure of the Frobenius algebras $V_q$ and $V_q^+$ remained mysterious. Indeed, the main applications of our work concentrated on the case $r=5$ where we had $V_q^+=\Q[x]/(x^2-x-1)$ for $q=\exp(i\pi/5)$ and $V_q^+=\Q[x]/(x^2+x+1)$ for $q=\exp(3i\pi/5)$ (here we have put $e_0=1$ and $e_2=x$). 

The main theorem of this article identifies $V_q^+$ with a completely different algebra, related to the geometry of two-bridge knots. I confess that I don't have a conceptual explanation of this result which came as a surprise from numerical experiments.  

Let $0<s<r$ be coprime odd integers and consider the two-bridge knot $K(r,s)$ with parameter $(r,s)$: it is characterized by the fact that its ramified double cover is the lens space $L(r,s)$, see \cite[Chapter 9]{Murasugi}. 
We set 
$$P(r,s)=\{\rho:\pi_1(S^3\setminus K(r,s))\to \SL_2(\C), \rho(m)\text{ is parabolic}\}/\SL_2(\C)$$
Here $m$ stands for any meridian of $K(r,s)$ and the quotient is by conjugation. This set is a $0$-dimensional algebraic variety defined over $\Q$. Hence $\Q[P(r,s)]$, the algebra of functions on $P(r,s)$ is a commutative $\Q$-algebra. We call it the Riley algebra as it is isomorphic to $\Q[x]/R(x)$ where $R$ is the Riley polynomial of $K(r,s)$, see Section \ref{Riley}.

\begin{theorem}\label{main}
Let $r,s$ be coprime odd integers satisfying $0<s<r$ and set $q=\exp(\frac{i\pi s}{r})$. There is an isomorphism of algebras $$\Q[P(r,s)]\simeq V_q^+.$$
\end{theorem}

When $s>1$, the knot $K(r,s)$ is hyperbolic and the monodromy of the hyperbolic structure on $S^3\setminus K(r,s)$ appears as a complex point of $P(r,s)$. This means that the trace field of $K(r,s)$ is a factor of $\Q[P(r,1)]$. On the contrary, the knot $K(r,1)$ is a torus knot, hence is not hyperbolic. However it has a parabolic representation with traces in $\Q(\cos(\frac{2\pi}{r}))$ which is then a factor of $\Q[P(r,1)]$. The Hermitian TQFT corresponding to $q=\exp(i\pi/r)$ is unitary and its signature is equal to its dimension. The Frobenius algebra in this case is the usual Verlinde algebra. 

In \cite{DM}, we observed without proof that amongst all roots of unity of order $r$, there is an involution corresponding to pairs of Frobenius algebras which are isomorphic as algebras. This phenomenon has now a nice explanation: it corresponds to the fact that $K(r,s)$ and $K(r,s^*)$ are isotopic as knots where $s^*$ is the inverse of $s$ modulo $r$. 

The reader may notice that the previous theorem only states an isomorphism of algebras. Indeed, there is no natural structure of Frobenius algebra on $\Q[P(r,s)]$. In the next theorem, we identify this structure in terms of Riley polynomials: for simplicity, we only state the case of $V_q$ in the introduction. 

\begin{theorem}\label{Frob}
In the settings of Theorem \ref{main}, the Frobenius algebra on $V_q$ has the following explicit form. 
Let $\epsilon_n=(-1)^{\lfloor ns/r\rfloor}$ for $n=1,\ldots,r-1$ and define the coprime and monic polynomials $P_{r-2},P_{r-1}$ of respective degrees $r-2,r-1$ such that the following holds:
$$\frac{P_{r-1}}{P_{r-2}}=X+
\cfrac{\epsilon_1\epsilon_2}
{X+\cfrac{\epsilon_2\epsilon_3}{X+\cfrac{\cdots}{\cdots+\cfrac{\epsilon_{r-2}\epsilon_{r-1}}{X}}}
}$$
Then $V_q=\Q[X]/(P_{r-1})$ and for any $f\in V_q$ one has: $$\epsilon(f)=\sum_{x}\Res_x\left(\frac{fP_{r-2}}{P_{r-1}}\right).$$
Finally, $P_{r-1}$ is related to the Riley polynomial by the formula $P_{r-1}(X)=R(X^2)$.
\end{theorem}
The last statement has already been obtained in \cite{JoKim}. This theorem shows that the Frobenius algebra structure involves the ``continued fractions structure'' of Riley polynomials. Its proof makes use of classical techniques of orthogonal polynomials. Moreover, the natural appearance of residues is a good sign for finding a spectral curve governing the cohomological field theories described above, which was the initial motivation of this work. 

%The Riley polynomial of $K(r,s)$ satisfies the equation $P_{r-1}(X)=R(X^2)$. 
In the SO$_3$ case, one can write $V_q^+=\Q[X]/(\chi)$ using the factorization $R(-X^2)=\chi(X)\chi(-X)$ which has also been obtained in \cite{JoKim}, see Section \ref{SO3}.
%In order to adapt Theorem \ref{Frob} to the SO$_3$ case, one proves the factorization $R(-X^2)=\chi(X)\chi(-X)$ which has also been obtained in \cite{JoKim}, see Section \ref{SO3}.. 
In the present article, this factorization follows from the properties of the element $e_{r-3}\in V_q^+$. This element has been used in \cite{BHMV} in order decompose the TQFT vector spaces associated to SU$_2$. It was also very useful in \cite{DM} in order to compute the first term in the $R$-matrix associated to the CohFT. A lot of questions remain open, for which we indicate some sparse results.

\subsubsection*{Simplicity}

For many values of $(r,s)$, the algebra $\Q[P(r,s)]\simeq V_q^+$ is simple, i.e. is a number field, shown in yellow in Figure \ref{table} (non-simple cases are shown in orange). The most appealing question is whether this algebra is always simple when $r$ is prime, for any value of $s$. This is well-known if the cyclotomic polynomial of order $r$ is irreducible modulo $2$ (because the algebra modulo $2$ is independent of $s$), the remaining cases are still open. It also seems that for a fixed even integer $k$, the algebra $\Q[P(r,r-k)]$ is simple provided that $r$ is big enough. We prove it in the case $k=2$ by analysing the roots of the polynomial $\chi$ defined above. 

\begin{proposition}\label{propsimple}
For any odd $r$, the algebras $\Q[P(r,r-2)]\simeq V_{-\exp(2i\pi/r)}^+$ are number fields. They have no real embeddings if $r\equiv 1[4]$ and one real embedding if $r\equiv -1[4]$.
\end{proposition}

\subsubsection*{Signatures}

The algebra $V_q^+$, beyond governing signatures of TQFTs, has its own signature denoted by $r_1(V_q^+)$. It is the signature of the bilinear form $(x,y)\mapsto \tr_{V_q^+}(xy)$ and is equal to the number of real embeddings of $V_q^+$ if it is a number field. It would be interesting to relate this signature to the signature of the form $\eta$ on $V_q$: this one is given by the formula $$\sg(\eta)=\sum_{n=1}^{r-1}\epsilon_n=2\sg(\eta^+)$$
which is also equal to the standard signature of the knot $K(r,s)$, see Section \ref{signature}.  
We will prove in Section \ref{signature} the inequality 
\begin{equation}\label{ineg}
|\sg(\eta^+)|\le r_1(V_q^+).
\end{equation}

I thought for a while that this inequality was an equality until counterexamples were found by P.V. Koseleff for $r\ge 39$. It would be nice to find an exact formula, maybe using Dedekind sums. Observe also that when $s>1$, the trace field of $K(r,s)$ has at least one complex embedding due to the hyperbolic structure. Hence $r_1(V_q^+)<\dim V_q^+=\frac{r-1}{2}$ in all those cases. 

\subsubsection*{Acknowledgments}
It is my pleasure to thank Pierre Charollois, Bertrand Deroin, Pascal Dingoyan, Pierre Godfard, Pierre-Vincent Koseleff and Gregor Masbaum for useful discussions around this work. 

\begin{figure}[htbp]\label{table}
\begin{center}
        \begin{tabular}{|c|c|c|c|c|c|c|c|c|c|c|c|}
        \hline
         $r\backslash s$ & $1$ & $3$ & $5$ & $7$ & $9$ & $11$ & $13$ & $15$ & $17$ & $19$ & $21$\\
        \hline
        $3$&\cellcolor{yellow}1&&&&&&&&&&\\
        \hline
        $5$&\cellcolor{yellow}2&\cellcolor{yellow}0&&&&&&&&&\\
        \hline
        $7$&\cellcolor{yellow}3&\cellcolor{yellow}1&\cellcolor{yellow}1&&&&&&&&\\
        \hline
        $9$&\cellcolor{orange}4&&\cellcolor{yellow}0&\cellcolor{yellow}0&&&&&&&\\
        \hline
        $11$&\cellcolor{yellow}5&\cellcolor{yellow}1&\cellcolor{yellow}1&\cellcolor{yellow}-1&\cellcolor{yellow}1&&&&&&\\
        \hline
        $13$&\cellcolor{yellow}6&\cellcolor{yellow}2&\cellcolor{yellow}0&\cellcolor{yellow}0&\cellcolor{yellow}2&\cellcolor{yellow}0&&&&&\\
        \hline
        $15$&\cellcolor{orange}7& & &\cellcolor{yellow}1& &\cellcolor{orange}1&\cellcolor{yellow} $1$&&&&\\
        \hline
        $17$&\cellcolor{yellow}8&\cellcolor{yellow}2&\cellcolor{yellow}2&\cellcolor{yellow}2&\cellcolor{yellow}0&\cellcolor{yellow}-2&\cellcolor{yellow}0&\cellcolor{yellow}0&&&\\
        \hline
        $19$&\cellcolor{yellow}9&\cellcolor{yellow}3&\cellcolor{yellow}1&\cellcolor{yellow}1&\cellcolor{yellow}1&\cellcolor{yellow}1&\cellcolor{yellow}3&\cellcolor{yellow}1&\cellcolor{yellow}1&&\\
        \hline
        $21$&\cellcolor{orange}10& &\cellcolor{yellow}2& &&\cellcolor{yellow}0&\cellcolor{orange}0&&\cellcolor{yellow}2&\cellcolor{yellow}0&\\
        \hline
        $23$&\cellcolor{yellow}11&\cellcolor{yellow}3&\cellcolor{yellow}1&\cellcolor{yellow}1&\cellcolor{yellow}-1&\cellcolor{yellow}1&\cellcolor{yellow}-1&\cellcolor{yellow}-3&\cellcolor{yellow}1&\cellcolor{yellow}1&\cellcolor{yellow}1\\
        \hline
        \end{tabular}
        \caption{Table of $\sg(\eta^+)$: yellow for simple cases, orange for non-simple ones. In all shown cases, the inequality \eqref{ineg} is an equality.}
\end{center}
\end{figure}
%\tableofcontents
\section{Generalities on Frobenius algebras}\label{generalites}

In this article, a Frobenius algebra is a finite dimensional commutative $\Q$-algebra $V$ endowed with a linear form $\epsilon:V\to \Q$ such that the bilinear form $\eta(x,y)=\epsilon(xy)$ is non-degenerate. 

There is a natural element $\Omega\in V$ constructed as follows: write $\eta^{-1}=\sum_{i}x_i\otimes y_i \in V\otimes V$. Then $\Omega=\sum_{i}x_i y_i$. In the standard equivalence between Frobenius algebras and $1+1$ TQFTs, $V$ is the image of the circle, $\epsilon$ corresponds to capping off with a disc and $\Omega$ corresponds to the punctured torus. 

The Frobenius algebra is said to be semi-simple if $V$ is semi-simple as an algebra, that is, isomorphic to a product of number fields. Here is a useful equivalent formulation: let $M_x\in \End(V)$ be the operation of multiplication by $x\in V$ and set $\tr_V(x)=\tr M_x$. The algebra $V$ is semi-simple if and only if the bilinear form $(x,y)\mapsto \tr_{V}(xy)$ is non-degenerate. In this case there is a unique invertible element $\alpha\in V$ such that $\epsilon(x)=\tr_V(\alpha x)$ for all $x\in V$ and one has moreover $\alpha^{-1}=\Omega$, see \cite[Section 5.1.1]{DM}.

The Frobenius algebras of this article arise in the following form. Let $P\in \Q[X]$ be a polynomial with simple roots and $h\in \Q(X)$ be a rational function without poles at roots of $P$. Then we set $V=\Q[X]/(P(X))$ and define for $f\in \Q[X]$:
$$\epsilon(f)=\sum_{z,P(z)=0}\Res_z\frac{fh}{P}.$$

This can be written equivalently 
$$\epsilon(f)=\sum_{z,P(z)=0} \frac{f(z)h(z)}{P'(z)}=\tr_{V/\Q}(\frac{fh}{P'})$$
so that one has $\alpha=\frac{h}{P'}$ and $\Omega=\frac{P'}{h}$.

In this case the invariant $\langle S_g\rangle$ associated to a genus $g$ surface is given by 

$$\langle S_g\rangle=\epsilon(\Omega^g)=\tr_{V}\Omega^{g-1}=\frac{1}{2i\pi}\int_C\frac{P'(z)^gdz}{P(z)h(z)^{g-1}}$$
where $C$ is a contour with index $1$ around the zeroes of $P$ and $0$ around the poles of $h$.

\begin{remark}
Notice that $h$ is part of the necessary data to define a Frobenius algebra. For instance, if one sets $h=1$, we get $\Omega=P'$ and $N(\Omega)=\operatorname{Disc}(P)$.
\end{remark}

%As the structure constants of $V$ are defined through residus, it would be tempting to call them ``residual". However, this adjective usually has another meaning. To justify the name, let us compute the norm of $\Omega$, that is the determinant of the multiplication by $\Omega$ on $V$. As it corresponds to the multiplication by $Q'$ on $\Q[X]/(Q)$ we have 

%$$N(\Omega)=\prod_{Q(x)=0}Q'(x)=\Res(P,P')=\Disc(P).$$

%If $V$ is a number field and $x$ is an integer, this number divides the discriminant of $V$ and is equal to it iff $x$ generates the integers.

\section{Signed Verlinde algebras for SU$_2$}\label{FrobSU2}

Let $q$ be a root of unity of order $2r$ with $r$ odd: in the sequel we will choose $q=\exp(i\pi s/r)$ with $0<s<r$ where $s$ is odd and prime to $r$. We use below the notation of \cite{BHMV}. 

We set $[n]=\frac{q^n-q^{-n}}{q-q^{-1}}$ and $[n]!=[n][n-1]\cdots [1]$. 
We will denote by $\epsilon_n\in \{-1,0,1\}$ the sign of $[n]$. We check that $\epsilon_0=\epsilon_r=0$, $\epsilon_1=1$ and 
$$\epsilon_n=\sg\left(\frac{\sin(\pi n s/r)}{\sin(\pi s/r)}\right)=(-1)^{\lfloor ns/r\rfloor}.$$
Moreover, as $q^r=-1$, we get $[r-n]=[n]$ and $\epsilon_n=\epsilon_{r-n}$.

Let $i,j,k\in \{0, \ldots, r-2\}$. The triple $(i,j,k)$ is said $r$-admissible if it satisfies 
\begin{equation}\label{r-adm}
i\le j+k, j\le i+k,k\le i+j, i+j+k \text{ is even and } i+j+k\le 2r-4.
\end{equation}

For such a triple, one writes $i=b+c, j=a+c, k=a+b$ and sets 
$$\langle i,j,k\rangle=(-1)^{a+b+c}\frac{[a+b+c+1]![a]![b]![c]!}{[a+b]![a+c]![b+c]!}.$$

We then define the Frobenius algebra $V_q$ as a $\Q$-vector space with basis $e_0,\ldots, e_{r-2}$ endowed with a symmetric bilinear form $\eta:V_q^2\to \Q$ for which $e_i$ is orthogonal and verifies $\eta(e_i,e_i)=\sg((-1)^i[i+1])=(-1)^i\epsilon_{i+1}$. We also endow it with a trilinear symmetric form $\omega:V_q^3\to \Q$ defined by $$\omega(e_i,e_j,e_k)=\sg \langle i,j,k\rangle\text{ if }(i,j,k)\text{ is }r\text{-admissible, }0\text{ otherwise.}$$

The product $\cdot:V_q\times V_q\to V_q$ defined by $\omega(x,y,z)=\eta(x\cdot y,z)$ endows $V_q$ with a structure of Frobenius algebra, as shown in \cite{DM}.

One checks by a direct computation the following formula where we put also $e_{-1}=e_{r-1}=0$:

$$e_1 e_n= -\epsilon_n\epsilon_{n+1}e_{n-1}+e_{n+1}.$$

It is natural to introduce polynomials which satisfy the same recursion formula, that is we set $P_0=1, P_1=X$ and 

\begin{equation}\label{rec}
XP_n=-\epsilon_n \epsilon_{n+1} P_{n-1}+P_{n+1}
\end{equation}

A direct induction shows that $P_n$ is even (resp.) odd if $n$ is even (resp. odd). These polynomials are obtained as the determinant of the upper-left minor of size $n$ of the following matrix of size $r-1$:
$$M=\begin{pmatrix} X & -\epsilon_1\epsilon_2 &  & &  \\
1& X & -\epsilon_2\epsilon_3 &  & \\
&&\ddots&&\\
&&1&X&-\epsilon_{r-2}\epsilon_{r-1}\\
&&&1&X
\end{pmatrix}$$

As $r-1$ is even, one has $P_{r-1}(X)=P_{r-1}(-X)$: this shows that $P_{r-1}$ is indeed the characteristic polynomial of the endomorphism of multiplication by $e_1$. This construction shows that the map $\Phi:\Q[X]\to V_q$ defined by $\Phi(P)=P(e_1)$ induces an algebra isomorphism
$$\Q[X]/(P_{r-1})\simeq V_q$$
which moreover satisfies $\Phi(P_n)=e_n$. For convenience, we will set $x=\Phi(X)=e_1$ so that we can write $e_n=P_n(x)$.

\begin{lemme}\label{semi-simple}
The polynomial $P_{r-1}$ has simple roots, equivalently, the algebra $V_q$ is semi-simple.
\end{lemme}
\begin{proof}
To start, recall that one can write $P_{r-1}(X)=R(X^2)$. It is sufficient to show that $R$ has simple roots: we will prove it by showing that its discriminant is odd. That is, we prove that the roots of $R$ in $\overline{\mathbb{F}}_2$ are distinct. 
We observe that the recursion relation \eqref{rec} defining $P_n$ is the same modulo 2 as the one we obtain for $s=1$ where $\epsilon_1=\cdots=\epsilon_{r-1}=1$. This gives the formula:
$$XU_n=U_{n+1}-U_{n-1}.$$
This polynomial satisfies the equation $U_n(i(t+t^{-1}))=i^n\frac{t^{n+1}-t^{-n-1}}{t-t^{-1}}$. Observe that one still has $U_{r-1}(x)=R(x^2)$ modulo $2$. As $R$ has degree $\frac{r-1}{2}$, we are reduced to finding $(r-1)/2$ distinct roots for $U_{r-1}$. 

As $r$ is odd, there are $r$ distinct  $r$-th roots of unity in $\overline{\mathbb{F}}_2$. As the Frobenius map $x\mapsto x^2$ is an isomorphism, there are as many roots of order $2r$, denoted by $1=\zeta_1,\ldots,\zeta_r$. Roots of $U_{r-1}$ have the form $i(\zeta_j+\zeta_j^{-1})$ for $j=2,\ldots,r$.  The involution $\zeta_j\mapsto \zeta_j^{-1}$ only fixes $\zeta_1=1$ and one has $\zeta_j+\zeta_j^{-1}=\zeta_k+\zeta_k^{-1}$ if and only if $j=k$. This finishes the proof of the lemma.
\end{proof}
\begin{remark}
This proof is a variant of the proof that the Riley polynomial has simple roots, see \cite{Riley}.
\end{remark}

\subsection{Decomposition into even and odd parts}

Denote by $V^+_q$ the subspace generated by the $e_n$ where $n$ is even. The Frobenius algebra structure on $V_q$ induces a Frobenius algebra structure on $V^+_q$ which we call the SO$_3$ signed Verlinde algebra of parameter $q$. 

\begin{remark}{\em
In \cite{DM}, it corresponds to the Frobenius algebra associated with SO$(3)$ with the root $-q$, which has order $r$.

Indeed, its basis is given by $e_0,e_2,\ldots, e_{r-3}$ and signs $\eta(e_{2i},e_{2i})=\epsilon_{2i+1}$. The formula for $\omega$ is the same. As changing $q$ to $-q$ does not change the quantum integers, the conclusion follows.}
\end{remark}

Writing the matrix $M$ in the basis $(e_0,e_2,\ldots,e_{r-3},e_1,e_3,\ldots,e_{r-2})$ we compute:
%$$A=\begin{pmatrix} \epsilon_0\epsilon_1&&&& \\ 1& \epsilon_2\epsilon_3&&&\\ & \ddots&\ddots&& \\ &&1&\epsilon_{r-5}\epsilon_{r-4}&\\ &&&1&\epsilon_{r-3}\epsilon_{r-2}\end{pmatrix},
%B=\begin{pmatrix} 1&\epsilon_1\epsilon_2&&& \\ &1& \epsilon_3\epsilon_4&&&\\ & &\ddots&\ddots&&\\&&&1&\epsilon_{r-4}\epsilon_{r-3} \\ &&&&1\end{pmatrix}$$
$$P_{r-1}(X)=\det(M)=\det \begin{pmatrix} X I & A \\ B& XI\end{pmatrix}=\det(X^2-AB).$$ 

On the other hand, $AB$ is precisely the matrix of $e_1^2$ acting on $V^+_q$. We get from this the equality $P_{r-1}(X)=R(X^2)$ where $R(t)=\det(t-e_1^2)|_{V_q^+}$ is the characteristic polynomial of $e_1^2$ acting on $V^+_q$. %As $e_1^2=e_2-\epsilon_2$ by \eqref{rec}, this shows that the characteristic polynomial of $e_2\in V^+_q$ is $R(X+\epsilon_2)$. 

%The good property of $e_2$ is that it generates the algebra $V^+_q$: its matrix in the standard basis is still tridiagonal but this time with a non-zero diagonal.

In the sequel, we will need other simple computations involving the first two and last two elements of the basis. 

\begin{lemme}\label{elements}
Let us write 
$$e_0=1,\quad e_1=x,\quad e_2=y,\quad e_{r-3}=\epsilon_2 w,\quad e_{r-2}=\iota.$$
One has the following formulas
$$x^2=y-\epsilon_2=-w^2, \quad \iota e_n=(-1)^n\epsilon_{n+1}e_{r-2-n},\quad \iota^2=-1.$$
\end{lemme}
\begin{proof}
This comes directly from the definition of the algebra structure so that we leave it to the reader. We simply observe that $(r-2,m,n)$ is $r$-admissible if and only if $m+n=r-2$ hence the product $e_{r-2}e_n$ only involves $e_{r-2-n}$ with a sign that can be computed explicitly.
\end{proof}

\begin{remark}
There is a nice analogy between the formula $V_q=V_q^+[\iota]$ and the more standard one: $\C=\R[i]$.  
\end{remark}

\subsection{The element $\Omega$}
Recall that any Frobenius algebra has a special element $\Omega$ which is the image of a punctured torus. It can be computed from any orthogonal basis $e_1,\ldots,e_n$ by the formula $$\Omega=\sum_{i=1}^n \frac{e_i^2}{\eta(e_i,e_i)}.$$

In the case of $V_q$ with its standard basis, this gives $\Omega=\sum_{n=0}^{r-2} (-1)^n\epsilon_{n+1} e_n^2$. The element corresponding to $V_q^+$ is $\Omega^+=\sum_{n=0}^{(r-3)/2}\epsilon_{2n+1}e_{2n}^2$. Separating even and odd terms in the sum and using Lemma \ref{elements} gives $\Omega=2\Omega^+$. Our purpose here is to give an alternative description: 

\begin{proposition} The element $\Omega$ satisfies the following equation:
$$\Omega=-P'_{r-1}(x)P_{r-2}(x).$$
\end{proposition}
\begin{proof}
This is a consequence of the Christoffel-Darboux formula for orthogonal polynomials, see \cite{Chihara}. We reproduce the (easy) proof for the convenience of the reader. 
Multiplying Equation \eqref{rec} by $P_n(Y)$ and exchanging the roles of $X$ and $Y$ we get:
$$\begin{cases}
XP_n(X)P_n(Y)=-\epsilon_n\epsilon_{n+1}P_{n-1}(X)P_n(Y)+P_{n+1}(X)P_n(Y)\\
YP_n(X)P_n(Y)=-\epsilon_n\epsilon_{n+1}P_{n-1}(Y)P_n(X)+P_{n+1}(Y)P_n(X)
\end{cases}$$
Multiplying by $\epsilon_{n+1}$ and taking the difference gives 
$$\epsilon_{n+1}(X-Y)P_n(X)P_n(Y)=Q_n(X,Y)+Q_{n-1}(X,Y)$$
 where $Q_n(X,Y)=\epsilon_{n+1}(P_{n+1}(X)P_n(Y)-P_{n+1}(Y)P_n(X))$. Summing the left hand side after multiplication by $(-1)^n$ hence gives a telescopic sum yielding the so-called Christoffel-Darboux formula:
 $$\sum_{n=0}^{r-2}(-1)^n\epsilon_{n+1} P_n(X)P_n(Y)=-\frac{P_{r-1}(X)P_{r-2}(Y)-P_{r-1}(Y)P_{r-2}(X)}{X-Y}$$
Letting $X$ go to $Y$ we get the so-called confluent form:
 $$\sum_{n=0}^{r-2}(-1)^n\epsilon_{n+1} P_n(X)^2=-(P'_{r-1}(X)P_{r-2}(X)-P_{r-1}(X)P_{r-2}'(X))$$

Mapping this equation to $V_q$ yields the proposition as $P_{r-1}$ goes to $0$. 
\end{proof}

As a consequence, we have $\Omega=-P'_{r-1}(x)P_{r-2}(x)=-\iota P'_{r-1}(x)$. 

In particular one can compute $\epsilon:\Q[X]/(P_{r-1})\to \Q$ with the formula 

$$\epsilon(f)=\tr_{V_q}(\Omega^{-1}f)=\tr_{V_q}(\frac{\iota f}{P'_{r-1}(x)})=\sum_{P_{r-1}(z)=0} \Res_z(\frac{P_{r-2} f}{P_{r-1}}).$$

This formula is particularly nice due to the following continued fraction expansion which follows directly from Equation \eqref{rec} by dividing by $P_n$ and applying induction: 

$$\frac{P_{r-1}}{P_{r-2}}=X+
\cfrac{\epsilon_1\epsilon_2}
{X+\cfrac{\epsilon_2\epsilon_3}{\cdots+\cfrac{\epsilon_{r-1}\epsilon_{r-2}}{X}}
}.$$

\subsection{Relation with 2-bridge knots}\label{Riley}
Recall that the fundamental group of the two-brige knot $K(r,s)$ where $r,s$ are coprime odd integers satisfying $0<s<r$ has a Wirtinger presentation given by 
$$G=\pi_1(S^3\setminus K(r,s))=\langle u,v|wu=vw\rangle, \quad w=u^{\epsilon_1}v^{\epsilon_2}\cdots u^{\epsilon_{r-2}}v^{\epsilon_{r-1}}.$$

Following Riley (see \cite{Riley} or \cite{KM}), any representation $\rho:G\to \SL_2(\C)$ such that $\rho(u)$ is parabolic can be conjugated such that 
$$\rho(u)=\begin{pmatrix} 1 & 1 \\ 0 & 1\end{pmatrix},\quad 
\rho(v)=\begin{pmatrix} 1 & 0 \\ x & 1\end{pmatrix}.$$
A direct computation shows that $\rho$ defined as above satisfies the relation defining $G$ if and only if the upper left entry of $\rho(w)$ vanishes. Taking $x$ a formal variable, we denote by $R(x)$ this coefficient and call $R$ the Riley polynomial of parameters $(r,s)$. This will not create a conflict of notation thanks to the following proposition.

\begin{proposition}
The polynomial $R$ defined above satisfies $$P_{r-1}(X)=R(X^2).$$
\end{proposition}
\begin{proof}
Multiplying the lines of $M$ by $\epsilon_i$ and conjugating by the diagonal matrix with entries $\epsilon_1,\epsilon_1\epsilon_2,\cdots$, we find 
$$P_{r-1}(X)=\epsilon_1\cdots\epsilon_{r-1}K_{r-1}(\epsilon_1 X,\ldots,\epsilon_{r-1}X)$$ where we denote by $K_n$ the continuant given by 
$$K_n(x_1,\ldots,x_n)=\det \begin{pmatrix} x_1 & -1 &  & &  \\
1& x_2 & -1 &  & \\
&&\ddots&&\\
&&1&x_{n-1}&-1\\
&&&1&x_n
\end{pmatrix}$$
The symmetry $\epsilon_i=\epsilon_{r-i}$ even implies that $\epsilon_1\epsilon_2\cdots\epsilon_{r-1}=1$.
The formula $K_n(x_1,\ldots,x_n)=x_1K_{n-1}(x_2,\ldots,x_n)+K_{n-2}(x_3,\ldots,x_n)$ and a direct recursion gives the identity

$$\begin{pmatrix}K_n(x_1,\ldots,x_n) & K_{n-1}(x_1,\ldots,x_{n-1})\\K_{n-1}(x_2,\ldots,x_n) & K_{n-2}(x_2,\ldots,x_{n-1})\end{pmatrix}=\begin{pmatrix} x_1 &1\\ 1 & 0 \end{pmatrix}\cdots \begin{pmatrix} x_{n-1} &1\\ 1 & 0 \end{pmatrix}\begin{pmatrix} x_{n} &1\\ 1 & 0 \end{pmatrix}.$$

Replacing $x_i$ with $\epsilon_i x$, we pack the matrices two by two and compute 
$$\begin{pmatrix} \epsilon_1 x & 1 \\ 1 & 0\end{pmatrix}
\begin{pmatrix} \epsilon_2 x & 1 \\ 1 & 0\end{pmatrix}=\begin{pmatrix} 1+\epsilon_1\epsilon_2 x^2 & \epsilon_1 x \\ \epsilon_2 x & 1\end{pmatrix}$$
Conjugating with $g=\begin{pmatrix} 1 & 0 \\ 0 & x\end{pmatrix}$ we can write it as follows:

$$g\begin{pmatrix} \epsilon_1 x & 1 \\ 1 & 0\end{pmatrix}
\begin{pmatrix} \epsilon_2 x & 1 \\ 1 & 0\end{pmatrix}g^{-1}=\begin{pmatrix} 1+\epsilon_1\epsilon_2 x^2 & \epsilon_1  \\ \epsilon_2 x^2 & 1\end{pmatrix}=
\begin{pmatrix} 1 & \epsilon_1 \\ 0 & 1\end{pmatrix}
\begin{pmatrix} 1& 0 \\ \epsilon_2 x^2 & 1\end{pmatrix}.$$

Putting all together gives 
$$g\begin{pmatrix}K_{r-1}(\epsilon_1x,\ldots,\epsilon_{r-1}x) & K_{r-2}(\epsilon_1x,\ldots,\epsilon_{r-2}x)\\K_{r-2}(\epsilon_2x,\ldots,\epsilon_{r-1}x) & K_{r-3}(\epsilon_2 x,\ldots,\epsilon_{r-2}x)\end{pmatrix}g^{-1}=\begin{pmatrix} 1 &1\\ 0 & 1 \end{pmatrix}^{\epsilon_1}\begin{pmatrix} 1 &0\\ x^2 & 1 \end{pmatrix}^{\epsilon_2}\cdots$$

Taking the upper left entry on both sides finally proves the proposition:
$$K_{r-1}(\epsilon_1x,\ldots,\epsilon_{r-1}x)=R(x^2)=P_{r-1}(x).$$
\end{proof}

%To sum up this section, let us give the main statement in the form of a proposition. 
%\begin{proposition}
%There is a representation $\rho:\pi_1(S^3\setminus K(r,s))\to \SL_2(V_q^+)$ given by 
%$$\rho(u)=\begin{pmatrix} 1 & 1 \\ 0 & 1\end{pmatrix},\quad 
%\rho(v)=\begin{pmatrix} 1 & 0 \\ e_1^2 & 1\end{pmatrix}.$$
%\end{proposition}

%It would be nice to have an interpretation of this representation in terms of TQFT, not just as the result of a computation.
\section{Signed Verlinde algebras for SO$_3$}\label{SO3}

Let us concentrate on the SO$_3$ case: one has $\Omega^+=\frac{1}{2}\Omega=-\frac{1}{2}P'_{r-1}(x)P_{r-2}(x)$. As $P_{r-1}(x)=R(x^2)$, $P_{r-2}(x)=\iota$ and $\iota x=-w$, we get 
$$\Omega^+=-xR'(x^2)\iota=wR'(-w^2).$$

It follows that $w$ plays a prominent role in $V_q^+$: let us define $\chi(t)=\det(t-w)|_{V_q^+}$, the characteristic polynomial of $w$. We recall that we have set $R(t)=\det(t-e_1^2)|_{V_q^+}$ and $x=e_1$. A simple computation shows that $w$ has the following matrix in the standard basis of $V_q^+$:

$$W=\begin{pmatrix}
&&&&\epsilon_{r-1}\\
&&&\epsilon_{r-3}&-\epsilon_{r-2}\\
&&\reflectbox{$\ddots$}&\reflectbox{$\ddots$}&\\
&\epsilon_4&-\epsilon_5&&\\
\epsilon_2& -\epsilon_3&&&
\end{pmatrix}$$
From the equality $w^2=-x^2$, we get $\det(t^2-w^2)=\det(t^2+x^2)$ hence $\det(t-w)\det(-t-w)=\det(-t^2-x^2)$. It follows that 
$$\chi(t)\chi(-t)=R(-t^2)=P_{r-1}(it).$$

Hence, differentiating the above equation we get $\Omega^+=-\frac{1}{2}\chi'(w)\chi(-w)$ in particular we get on $V_q^+$:
$$\epsilon(f)=-2\sum_{\chi(t)=0}\Res_t \frac{f(t)}{\chi(t)\chi(-t)}=-2\sum_{\chi(t)=0}\Res_t\frac{f(t)}{R(-t^2)}.$$

As one has $V_q^+=\Q[t]/(\chi(t))$, it is useful to study the polynomial $\chi$. Its constant coefficient is $\pm 1$, showing that $w$ is a unit. Moreover, numerical experiments show that it has quite small coefficients that we compute in the next section.

\subsection{Two examples}

\subsubsection{The case s=1}
Set $q=\exp(\frac{i\pi}{r})$. In this case, as explained in Lemma \ref{semi-simple}, one has $P_{r-1}=U_{r-1}$ where $U_n(i(t+t^{-1}))=i^n\frac{t^{n+1}-t^{-n-1}}{t-t^{-1}}$. This proves that $V_q$ is the sub-algebra of $\Q[t]/(t^{2r}-1)$ generated by $i(t+t^{-1})$. In particular, the list of embeddings of $V_q$ in $\C$ is given by $x\mapsto 2i\cos(\frac{k\pi}{r})$ for $k=1,\ldots,r-1$. This also shows that $V_q^+$ is totally real.

%As $U_{r-1}(x)=0$, we get $x=i(t+t^{-1})$ with $t^{2r}=1$: hence $V_q=\Q(i\cos(\frac{\pi}{r}))=\Q(i,\cos(\frac{\pi}{r}))$.

A direct computation gives $P_{r-2}(x)=i^{r-2}$ and $P'_{r-1}(x)=\frac{2r i^r}{(t-t^{-1})^2}$, hence: 
$$\Omega=\frac{-2r}{(t-t^{-1})^2}.$$

From the formulas of Section \ref{generalites}, we get 
$$\langle S_g\rangle=\tr_{V_q}(\Omega^{g-1})=\big(\frac r 2 \big)^{g-1}\sum_{k=1}^{r-1}\sin(\frac{k\pi}{r})^{2-2g}.$$

We recover the standard Verlinde formula, as the TQFT for $q=\exp(i\pi/r)$ is unitary and its signature coincides with its dimension. We also observe that $V_q$ is simple if and only if $r$ is prime. This is due to the decomposition $t^{2r}-1=\prod_{d|2r}\Phi_d$ where $\Phi_d$ are the cyclotomic polynomials. 

\subsection{The case $s=r-2$}

This case is opposite to the previous one and behave somewhat more simply than the general case so that we include an analysis of it, yielding a proof of Proposition \ref{propsimple}. Set $q=\exp(\frac{i\pi(r-2)}{r})=-\exp(\frac{2i\pi}{r})$. 

\begin{proposition}
Setting $\chi(t)=\det(t-W)$, the characteristic polynomial of $W$ one has:
\begin{enumerate}
\item $\tr(W)=1$.
\item The eigenvalues of $W$ have positive real parts. 
\item The polynomial $\chi$ is irreducible over $\Q$.
\item $\chi$ have no real roots if $r=1[4]$ and one real root if $r=-1[4]$.
\end{enumerate}
\end{proposition}
\begin{proof}
For this specific root $q$, the signs $\epsilon_1,\ldots,\epsilon_{r-1}$ are alternating except for the two middle ones: this gives a very simple matrix $W$ for $w$ acting on $V_q^+$: if $r=4n+1$, the first $n$ colums of $W$ have $-1$s, the last $n$ columns have $1$s. If $r=4n+3$, the first $n+1$ rows have $1$s, the last $n$ rows have $-1$s.

It follows that $\tr(W)=1$. Moreover, $W$ is almost antisymmetric: $M+M^T=2E$ where $E$ has only one non zero entry at $(n+1,n+1)$ where it is equal to $1$. 

Let $X$ be a normalized eigenvector for $W$, i.e. $X^*X=1$ where we write $X^*=\overline{X}^T$. One has $X^*WX=\lambda$ and $X^*W^*X=\overline{\lambda}$. Summing the two gives 

$$\lambda+\overline{\lambda}=2X^*EX=2|X_{n+1}|^2.$$ 
This gives the positivity of the real part of $\lambda$. Moreover, if $X_{n+1}=0$, one can solve the linear system $WX=\lambda X$ step by step and prove that $X=0$. Hence $\re(\lambda)>0$ as claimed. 

Suppose that there is a non trivial decomposition $\chi=PQ$. As $\chi$ is monic, $P$ and $Q$ are also monic with integral coefficients. One has $1=\sum_{P(\lambda)=0}\re(\lambda)+\sum_{Q(\lambda)=0}\re(\lambda)$. The two terms are strictly positive integers summing to 1, a contradiction. 

Let us prove the last point: using Descartes' rule of signs and the previous point, it is sufficient to show that the characteristic polynomial $\chi(t)$ has coefficients with alternating signs. The computation of those coefficient is deferred to Lemma \ref{explicit} stated below. It remains to determine the sign of the contribution of the pair $(A,B)$ to the characteristic polynomial. 
As $W_{ij}+W_{ji}=0$ unless $i=j$, we find that if $|A|$ is even, its sign is $|A|/2$ whereas if $|A|$  is odd, its sign is $(|A|-1)/2$, the same being true for $B$. Setting $a=|A|$ and $b=|B|$, one has $a+b=k$ and the sign of the contribution is 
$$(-1)^{k(k+1)/2+a(a-1)/2+b(b-1)/2}.$$
This can be rewritten $(-1)^{a+b+ab}$. However, because of the conditions on $A$ and $B$, $a$ and $b$ cannot be both odd integers so that $(-1)^{ab}=1$ and we get the sign $(-1)^k$ as expected. 

\end{proof}

\begin{lemme}\label{explicit}
Let $M$ be a matrix of size $n$ with $M_{i,j}=0$ if $i+j\notin\{n+1,n+2\}$. Its characteristic polynomial has the following expression:
$$\det(t-M)=\sum_{k=0}^n t^{n-k}(-1)^{\frac{k(k+1)}{2}}\sum_{(A,B)}\prod_{(i,j)\in A\cup B} M_{ij}$$
In this formula, $A\subset \{(i,j), i+j=n+1\}$ and $B\subset \{(i,j), i+j=n+2\}$ satisfy the following properties. 
\begin{enumerate}
    \item They are both invariant by the involution $(i,j)\mapsto (j,i)$.
    \item Their images by the projections $(i,j)\mapsto i$ and $(i,j)\mapsto j$ are disjoint.
    \item The sum of their cardinality is equal to $k$.
\end{enumerate}
\end{lemme}
\begin{proof}
The proof follows from a direct analysis of the expansion $\det(t-M)=\sum_{\sigma\in S_n}\epsilon(\sigma)\prod_{i=1}^n(t\delta_{i\sigma(i)}-M_{i\sigma(i)})=\sum_{\sigma\in S_n}\epsilon(\sigma)\sum_{I, J}t^{|I|}(-1)^{|J|}\prod_{j\in J}M_{j\sigma(j)}$. In the formula, $I,J$ run over all partitions of $\{1,\ldots,n\}$ such that $\sigma(i)=i$ for all $i\in I$. %For the corresponding term to be non-zero, $J$ contains at most one fixed point of $\sigma$ which is also a fixed point of $\tau(i)=n+1-i$. 
By construction, $\sigma$ permutes $J$. Writing $\sigma(j)+j=n+1+\delta_j$ for $\delta_j\in\{0,1\}$, one finds that for any $j_1<j_2\in J$: 
$$\sigma(j_2)-\sigma(j_1)=j_1-j_2+\delta_{j_2}-\delta_{j_1}< 0.$$
This proves that $\sigma$ invert all pairs in $J$ hence $\sigma$ acts by puting all elements in $J$ in reverse order. In particular $\sigma$ is an involution and $\epsilon(\sigma|_J)=(-1)^{k(k-1)/2}$ where $k=|J|$. The final sign of the contribution of the pair $(I,J)$ is $(-1)^{k+k(k-1)/2}$ as expected. The last point to prove is the fact that $A$ and $B$ are stable by the involution $(i,j)\mapsto (j,i)$ but this comes from the fact that $\sigma$ is an involution.
\end{proof}
\subsection{On signatures}\label{signature}

In this section we compare various notions of signatures that naturally appear in this context.
Let $V$ be a finite dimensional commutative $\Q$-algebra: for $t\in V$ we define $\sg_V(t)$ to be the signature of the quadratic form 
$$(x,y)\mapsto \tr_{V}(xyt).$$ 

If $V$ is a number field, let $\phi_1,\ldots,\phi_{r_1}:V\to \R$ be the family of real embeddings of $V$. We have the formula 
$$ \sg_V(t)=\sum_{i=1}^{r_1} \sg \phi_i(t).$$
This gives in particular
$$\sg_V(1)=r_1(V)\text{ and }
|\sg_V(t)|\le r_1\text{ for all }t\in V.$$
As one can write $\eta(x,y)=\tr_V(\Omega^{-1}xy)$, the usual signature of the bilinear form $\eta$ is given by $$\sg(\eta)=\sg_V(\Omega^{-1})=\sg_V(\Omega).$$

Recall that the Frobenius algebra $V_q$ and its even part $V_q^+$ are naturally endowed with a non-degenerate bilinear form $\eta$ (resp. $\eta^+$) which have associated signatures. From the fact that the standard basis is orthogonal we immediately compute 
$$\Sign(\eta)=\sum_{n=1}^{r-1}\epsilon_n=2\Sign(\eta^+)$$

On the other hand, the knot $K(r,s)$ also has a signature: it is by definition the signature of the matrix $M+M^T$ where $M$ is a Seifert matrix of $K(s,r)$. One finds in \cite[Theorem 9.3.6]{Murasugi} the formula 
$$\Sign(K(r,s))=\sum_{n=1}^{r-1}\epsilon_n$$
so that these signatures coincide.


\begin{thebibliography}{10}

\bibitem[BHMV]{BHMV}
C. Blanchet, N. Habegger, G. Masbaum and P.Vogel,
\newblock Topological Quantum Field Theories derived from the Kauffman bracket.
\newblock \emph{Topology}, \textbf{34}, no. 4, 883-927, 1995.

\bibitem[C]{Chihara}
T. S. Chihara,
\newblock An introduction to orthogonal polynomials.
\newblock Gordon and Breach, Science Publishers, 1978.

\bibitem[DM]{DM}
B. Deroin and J. Marché,
\newblock Toledo invariants of Topological Quantum Field Theories.
\newblock \emph{arXiv:2207.09952}

\bibitem[JK]{JoKim}
K. Jo and H. Kim,
\newblock Continuant, Chebyshev polynomials and Riley polynomials.
\newblock \emph{Journal of knot theory and its ramifications}, Vol. 31, No.1, 2022.

\bibitem[KM]{KM}
T. Kitano, T. Morifuji.
\newblock A note on Riley polynomials of 2-bridge knots.
\newblock \emph{Ann. Fac. Sci. Toulouse}, Math. (6) 26, No. 5, 1211-1217 (2017).

\bibitem[M]{Murasugi}
K. Murasugi,
\newblock Knot theory and its applications,
\newblock \emph{Birkhaüser}.

\bibitem[R]{Riley}
R. Riley,
\newblock Parabolic representations of knot groups. I.
\newblock \emph{Proc. Lond. Math. Soc.}, III. Ser. 24, 217-242 (1972).

\end{thebibliography}
\end{document}